\documentclass[12pt,reqno]{amsart}

\usepackage{multirow}
\usepackage{hhline}

\usepackage{mathtools}
\usepackage{times}
\usepackage[T1]{fontenc}
\usepackage{mathrsfs}
\usepackage{latexsym}
\usepackage[dvips]{graphics}
\usepackage[titletoc, title]{appendix}
\usepackage{amsmath,amsfonts,amsthm,amssymb,amscd}
\usepackage[dvipsnames]{xcolor}
\usepackage{hyperref}
\usepackage{amsmath}
\usepackage[utf8]{inputenc}

\usepackage{color}
\usepackage{breakurl}

\usepackage{comment}
\newcommand{\bburl}[1]{\textcolor{blue}{\url{#1}}}

\newtheorem{thm}{Theorem}[section]

\newtheorem{cor}[thm]{Corollary}

\newtheorem{lem}[thm]{Lemma}
\newtheorem{prop}[thm]{Proposition}

\newtheorem{rek}[thm]{Remark}

\usepackage[utf8]{inputenc}

\DeclareFixedFont{\ttb}{T1}{txtt}{bx}{n}{12} 
\DeclareFixedFont{\ttm}{T1}{txtt}{m}{n}{12}  

\usepackage{color}
\definecolor{deepblue}{rgb}{0,0,0.5}
\definecolor{deepred}{rgb}{0.6,0,0}
\definecolor{deepgreen}{rgb}{0,0.5,0}

\usepackage{listings}

\newcommand\pythonstyle{\lstset{
language=Python,
basicstyle=\ttm,
morekeywords={self},              
keywordstyle=\ttb\color{deepblue},
emph={MyClass,__init__},          
emphstyle=\ttb\color{deepred},    
stringstyle=\color{deepgreen},
frame=tb,                         
showstringspaces=false
}}

\lstnewenvironment{python}[1][]
{
\pythonstyle
\lstset{#1}
}
{}

\newcommand\pythoninline[1]{{\pythonstyle\lstinline!#1!}}

\numberwithin{equation}{section}

\DeclareFontFamily{U}{mathx}{}
\DeclareFontShape{U}{mathx}{m}{n}{<-> mathx10}{}
\DeclareSymbolFont{mathx}{U}{mathx}{m}{n}
\DeclareMathAccent{\widehat}{0}{mathx}{"70}
\DeclareMathAccent{\widecheck}{0}{mathx}{"71}

\begin{document}

\title{The Best Two-Term Underapproximation Using Numbers From Fibonacci-Type Sequences}

\author{Shiliaev Mark}

\begin{abstract}
    This paper studies the greedy two-term underapproximation of $\theta\in (0,1]$ using reciprocals of numbers from a Fibonacci-type sequence $(c_n)_{n=1}^\infty$. We find the set of $\theta$ whose greedy two-term underapproximation is the best among all two-term underapproximations using $1/c_n$'s. We then derive a neat description of the set when $(c_n)_{n=1}^\infty$ is the Fibonacci sequence or the Lucas sequence. 
\end{abstract}

\email{\textcolor{blue}{\href{mailto:mshiliaev@tamu.edu}{mshiliaev@tamu.edu}}}
\address{Department of Mathematics\\ Texas A\&M University, College Station, TX 77843, USA}

\subjclass[2020]{11A67, 11B99}

\keywords{Two-term underapproximation; greedy algorithm; Fibonacci sequence; Egyptian fraction}

\thanks{The author would like to thank Dr. H. V. Chu, Department of Mathematics, Texas A\&M University, for introducing the author to this project and for his feedback on the research progress.} 

\maketitle

\keywords{}

\maketitle


\section{Introduction and main results}

Let $\theta$ be a number in $(0, 1]$. An \textit{underapproximation} of $\theta$ is a series $\sum_{n=1}^{\infty}1/a_{n}$, where $a_n$'s are positive integers and $\sum_{n=1}^{\infty}1/a_{n} \leq \theta$. In this paper, we shall restrict ${a_n}$'s to a subset $A$ of the natural numbers. In particular, a \textit{greedy} underaproximation of $\theta$ over a set $A \subseteq \mathbb{N}$ is formed stepwise, where at each step, we pick the smallest number $n$ in ${A}$ such that 
 $1/n$ is smaller than the remainder after the previous step.
Formally, we define a function $G: (0, 1] \rightarrow \mathbb{N} \mbox{ as } G(\theta) =  \lfloor 1/\theta \rfloor .$ We then obtain a sequence  $(c_n)_{n=1}^\infty$ recursively as follows: 
\\
$$c_1 = \min\{a \in A: a> G(\theta)\} \mbox{ and } c_n= \min \left\{ a \in A: a>G\left(\theta - \sum_{m=1}^{n-1}\frac{1}{c_m}\right)\right\}.$$ Consequently, the greedy underapproximation of $\theta$ over $A$ is $\sum_{n=1}^\infty1/c_n$

If $A = \mathbb{N}$, we get the classical greedy underapproximations, which have been studied in \cite{1, be, 6, 3, d, lls, n}. A natural question is whether the greedy algorithm gives the best finite underapproximation $\sum_{n=1}^{m}1/c_n, m \in \mathbb{N}$, at each step. One may focus on two-term underapproximations as in \cite{4, 2, n}, since they are the first non-trivial approximations resulting from the greedy process. Surprisingly, two-term underapproximations can be used to prove results about the long-term behavior of greedy underapproximations. For example, Kovač \cite{2} proved that the Lebesgue measure of all real numbers whose finite greedy underapproximations, are eventually the best is 0. 

If $A$ consists of powers of 2, then the greedy process would give the best finite underapproximation out of all finite underapproximations using integers from $A$ at every step, due to how sparse the set $A$ is. We would like to know what happens if $A$ is less dense than $\mathbb{N}$ but denser than the powers of 2. A good candidate for an investigation is the Fibonacci sequence $(F_n)_{n=0}^\infty$ defined as: $F_0 = 0$, $F_1 = 1$, $F_{n} = F_{n-1} +F_{n-2}$ for $n\ge 2$. We shall consider more general integer sequences $(a_n)_{n=0}^\infty$ satisfying:

\begin{enumerate}
\item[a)] $a_0 > 0$, $a_1\ge 1$, $a_0 \leq a_1$, 
\item[b)] $a_{n} = a_{n-1} + a_{n-2}$ for $n\ge 2$,
\item[c)] $\chi :=  a_0^2+a_1a_0-a_1^2 > 0$.
\end{enumerate}
From Condition a) and b), we obtain 
$$1\le a_1\ \ <\ a_2 \ <\ a_3\ <\ a_4 \ <\ \cdots.$$

For $\theta\in (0,1]$, let $g_1(\theta)$ be the smallest index $n \ge 1$ such that  
$$\frac{1}{a_{g_1(\theta)}}\ <\ \theta.$$

Let $g_2(\theta)$ be the smallest index $n \ge g_1(\theta)$ such that 
$$\frac{1}{a_{g_2(\theta)}}\ <\ \theta - \frac{1}{a_{g_1(\theta)}}.$$
Set 
$$\mathcal{G}(\theta)\ :=\ \frac{1}{a_{g_1(\theta)}} +\frac{1}{a_{g_2(\theta)}},$$
called the greedy two-term underapproximation of $\theta$ using terms from the sequence $a_1, a_2, a_3, \ldots$.
Our goal is to determine when $\mathcal{G}(\theta)$ is the best two-term underapproximation of $\theta$ out of all possible underapproximations from the collection
$$\mathcal{A}(\theta, 2)\ =\ \left\{\frac{1}{a_m} +\frac{1}{a_n}\,:\, 1\le m \leq n\mbox{ and }\frac{1}{a_m} +\frac{1}{a_n} < \theta\right\}.$$

\begin{thm}\label{m1} Let $\theta \in (0, 1]$.
The greedy two-term underapproximation $\mathcal{G}(\theta)$ is the best underapproximation out of $\mathcal{A}(\theta, 2)$ if and only if 
$$\theta \ \notin\ \bigcup_{n = 0}^{\infty}\left(\frac{1}{a_{2n+3}}+\frac{1}{a_{2n+4}}, \frac{1}{a_{2n+2}} + \frac{1}{a_{2n+3 + \xi(n)}}\right],$$
where $\xi(n)$ is the largest nonnegative integer such that
\begin{equation}\label{e6}a_{2n+2}F_{\xi(n)} +a_{2n+3}F_{\xi(n)+1}\ \le\ \frac{a_{2n+2}a_{2n+3}a_{2n+4}}{a_0^2+a_1a_0-a_1^2}.\end{equation}
\end{thm}

\begin{rek}\normalfont
The function $\xi: \mathbb{N}\cup \{0\}\rightarrow \mathbb{N}\cup \{0\}$ is well-defined because for $n\ge 0$
\begin{align*}
    \frac{a_{2n+2}a_{2n+3}a_{2n+4}}{a_0^2+a_1a_0-a_1^2}\ \ >\ \frac{a_{1}a_{2n+3}a_{3}}{a_0^2+a_1a_0-a_1^2}&\ =\ \frac{a_{1}(2a_{1}+a_0)a_{2n+3}}{a_0^2+a_1a_0-a_1^2}\\
    &\ =\ \frac{(2a^2_{1}+a_1a_0)a_{2n+3}}{a_0^2+a_1a_0-a_1^2}\\
    &\ >\ a_{2n+3}\ =\ a_{2n+2}F_0 + a_{2n+3}F_1.
\end{align*}

\end{rek}

This paper is structured as follows: Section 2 outlines all preliminary results, Section 3 proves Theorem 1.1, while Section 4 applies it to some sequences to find a simpler way of defining $\theta(n)$ for them.
\section{Preliminary results}

\begin{lem}\label{l1}
For $m, n\ge 0$, we have
$$a_{n+m} \ =\ F_{n-1}a_m + F_{n}a_{m+1}.$$
\end{lem}

\begin{proof}
Recall the well-known identity, whose proof is in the Appendix: 
\begin{equation}\label{e1}F_{n+m} \ =\ F_{n-1}F_m +F_n F_{m+1}, \mbox{ for all }m,n\in\mathbb{Z}.\end{equation}
Due to the recurrence relation of $(a_n)_{n=0}^\infty$, we have
\begin{equation}\label{e2}a_n\ =\ a_0 F_{n-1} + a_1 F_n, \mbox{ for all }n\ge 0,\end{equation}
which can be easily proved by induction. Choose $m, n\ge 0$. By \eqref{e1} and \eqref{e2}, 
\begin{align*}
    a_{n+m} &\ =\ a_0 F_{n+m-1} +a_1 F_{n+m}\\
    &\ =\ a_0 (F_{n-1}F_{m-1} + F_n F_m) + a_1(F_{n-1}F_m + F_nF_{m+1})\\
    &\ =\ F_{n-1}(a_0F_{m-1} + a_1F_m) + F_n(a_0F_m + a_1F_{m+1})\\
    &\ =\ a_mF_{n-1} + a_{m+1}F_n.
\end{align*}
\end{proof}

\begin{lem}
For $n\ge 0$, we have
\begin{equation}\label{e3}a_n a_{n+3} - a_{n+1}a_{n+2}\ =\ (-1)^{n}\chi.\end{equation}
\end{lem}

\begin{proof}
We proceed by induction on $n$. Base case: for $n = 0$, 
$$a_0 a_3 - a_1a_2\ =\ a_0(2a_1+a_0) - a_1(a_0+a_1)\ =\ a^2_0+a_0a_1-a_1^2\ =\ \chi.$$
Inductive hypothesis: suppose that \eqref{e3} holds for $n = k$ for some $k\ge 0$. We have
\begin{align*}
    a_{k+1}a_{k+4} - a_{k+2}a_{k+3}&\ =\ a_{k+1}(a_{k+3}+a_{k+2}) - (a_{k+1}+a_k)a_{k+3}\\
    &\ =\ -(a_{k}a_{k+3}-a_{k+1}a_{k+2})\\
    &\ =\ -(-1)^{k}\chi\ =\ (-1)^{k+1}\chi.
\end{align*}
This completes our proof. 
\end{proof}

\begin{cor}\label{c3}
For $n\ge 0$, $\xi(n)$ is the largest nonnegative integer such that
\begin{equation}\label{e10}\frac{1}{a_{2n+3+\xi(n)}}\ \ge\ \frac{\chi}{a_{2n+2}a_{2n+3}a_{2n+4}}.\end{equation}
Equivalently,
\begin{equation}\label{e11}\frac{1}{a_{2n+3}} +\frac{1}{a_{2n+4}}\ \le\ \frac{1}{a_{2n+2}} + \frac{1}{a_{2n+3+\xi(n)}}.\end{equation}
\end{cor}

\begin{proof}
To prove \eqref{e10}, we use \eqref{m1} and Lemma \ref{l1}. To prove \eqref{e11}, use \eqref{e3} to obtain
\begin{align*}
\frac{1}{a_{2n+3}} - \frac{1}{a_{2n+2}}+ \frac{1}{a_{2n+4}}\ =\ \frac{1}{a_{2n+4}}-\frac{a_{2n+1}}{a_{2n+2}a_{2n+3}}&\ =\ \frac{a_{2n+2}a_{2n+3}-a_{2n+1}a_{2n+4}}{a_{2n+2}a_{2n+3}a_{2n+4}}\\
&\ =\ \frac{\chi}{a_{2n+2}a_{2n+3}a_{2n+4}}\\
&\ \le\ \frac{1}{a_{2n+3+\xi(n)}}.
\end{align*}
\end{proof}

\begin{cor}\label{c1}
For $n\ge 0$, 
$$\frac{1}{a_{2n+1}} - \frac{1}{a_{2n+2}} - \frac{1}{a_{2n+3}} \ >\ 0.$$
\end{cor}

\begin{proof}
By \eqref{e3}, we have
\begin{align*}
    \frac{1}{a_{2n+1}} - \frac{1}{a_{2n+2}} - \frac{1}{a_{2n+3}}&\ =\ \frac{a_{2n+2}-a_{2n+1}}{a_{2n+1}a_{2n+2}} - \frac{1}{a_{2n+3}}\\
    &\ =\ \frac{a_{2n+2}a_{2n+3}-a_{2n+1}a_{2n+3} - a_{2n+1}a_{2n+2}}{a_{2n+1}a_{2n+2}a_{2n+3}}\\
    &\ =\ \frac{a_{2n}a_{2n+3} - a_{2n+1}a_{2n+2}}{a_{2n+1}a_{2n+2}a_{2n+3}}\\
    &\ =\ \frac{\chi}{a_{2n+1}a_{2n+2}a_{2n+3}} \ >\ 0.
\end{align*}
\end{proof}

\begin{cor}\label{c2}
For $n\ge 0$,
$$\frac{1}{a_{2n+3}} + \frac{1}{a_{2n+4}} - \frac{1}{a_{2n+2}}\ >\ 0.$$
\end{cor}

\begin{proof}
From the proof of Corollary \ref{c3}, 
$$\frac{1}{a_{2n+3}} - \frac{1}{a_{2n+2}}+ \frac{1}{a_{2n+4}}\ =\ \frac{\chi}{a_{2n+2}a_{2n+3}a_{2n+4}}\ > \ 0.$$
\end{proof}

\begin{lem}\label{l2}
For $n\ge 0$, 
$$\frac{1}{a_{2n+2}} - \frac{1}{a_{2n+3}} - \frac{1}{a_{2n+5}} \ >\ 0.$$
\end{lem}

\begin{proof}
Write
\begin{align*}
    \frac{1}{a_{2n+2}} - \frac{1}{a_{2n+3}} - \frac{1}{a_{2n+5}}&\ =\ \frac{a_{2n+1}}{a_{2n+2}a_{2n+3}} - \frac{1}{a_{2n+5}}\\
    &\ =\ \frac{a_{2n+1}a_{2n+5}- a_{2n+2}a_{2n+3}}{a_{2n+2}a_{2n+3}a_{2n+5}}\\
    &\ =\ \frac{a_{2n+1}(a_{2n+3}+a_{2n+4})- (a_{2n+1}+a_{2n})a_{2n+3}}{a_{2n+2}a_{2n+3}a_{2n+5}}\\
    &\ =\ \frac{a_{2n+1}a_{2n+4} - a_{2n}a_{2n+3}}{a_{2n+2}a_{2n+3}a_{2n+5}}\\
    &\ =\ \frac{a_{2n+1}a_{2n+2} + a_{2n+3}(a_{2n+1}-a_{2n})}{a_{2n+2}a_{2n+3}a_{2n+5}} > 0.
\end{align*}

$$ \frac{1}{a_{2n+2}} - \frac{1}{a_{2n+3}} - \frac{1}{a_{2n+5}} \ >\ 0,$$
as desired. 
\end{proof}

\section{Proof of Theorem \ref{m1}}

\begin{proof}[Proof of Theorem \ref{m1}]
Suppose that for some $m\ge 0$, we have
\begin{equation}\label{e4}\frac{1}{a_{2m+3}}+\frac{1}{a_{2m+4}}\ <\ \theta \ \le\ \frac{1}{a_{2m+2}} + \frac{1}{a_{2m+3 + \xi(m)}}.\end{equation}
We shall show that $\mathcal{G}(\theta)$ is not the best two-term underapproximation out of $A(\theta,2)$.

By Corollaries \ref{c1} and \ref{c2}, 
$$\frac{1}{a_{2m+2}}\ <\ \frac{1}{a_{2m+3}}+\frac{1}{a_{2m+4}}\ <\ \theta \ \le\ \frac{1}{a_{2m+2}} + \frac{1}{a_{2m+3 + \xi(m)}}\ <\ \frac{1}{a_{2m+1}}.$$
Hence, $g_1(\theta) = 2m+2$. It then follows from \eqref{e4} and Corollary \ref{c3} that 
$$\mathcal{G}(\theta) \ \le\ \frac{1}{a_{2m+2}} + \frac{1}{a_{2m+4+\xi(m)}}\ <\ \frac{1}{a_{2m+3}} +\frac{1}{a_{2m+4}} \ <\ \theta,$$
so $\mathcal{G}(\theta)$ is not the best two-term underapproximation.

Conversely, suppose that $\mathcal{G}(\theta)$ is not the best two-term underapproximation of $\theta$. Let $g_1(\theta) = k$. If $k$ is odd,  then Corollary \ref{c1} gives
$$\frac{1}{a_k} \ >\ \frac{1}{a_{k+1}} +\frac{1}{a_{k+2}},$$
making $\mathcal{G}(\theta)$ the best. Hence, $k$ must be even, i.e., $k = 2m+2$ for some $m\ge 0$. Since $\mathcal{G}(\theta)$ is not the best, there are $\ell_2 > \ell_1 > 2m+2$ such that 
\begin{equation}\label{e7}\frac{1}{a_{2m+2}} \ <\ \mathcal{G}(\theta) \ <\ \frac{1}{a_{\ell_1}} +\frac{1}{a_{\ell_2}}\ <\ \theta.\end{equation}
By Lemma \ref{l2}, 
$$\frac{1}{a_{2m+3}} +\frac{1}{a_{2m+5}}\ <\ \frac{1}{a_{2m+2}} \ <\ \frac{1}{a_{\ell_1}} +\frac{1}{a_{\ell_2}}\ <\ \theta.$$
It follows that $(\ell_1, \ell_2) = (2m+3, 2m+4)$, and 
$$\frac{1}{a_{2m+3}} +\frac{1}{a_{2m+4}}\ <\ \theta.$$
By Corollary \ref{c3}, 
\begin{equation}\label{e8}\frac{1}{a_{2m+3}} + \frac{1}{a_{2m+4}}\ \le\ \frac{1}{a_{2m+2}} + \frac{1}{a_{2m+3+\xi(m)}}.\end{equation}
If 
$\theta  > a^{-1}_{2m+2} + a^{-1}_{2m+3+\xi(m)}$, then
$$\frac{1}{a_{\ell_1}} +\frac{1}{a_{\ell_2}}\ =\ \frac{1}{a_{2m+3}} + \frac{1}{a_{2m+4}} \ \le\ \frac{1}{a_{2m+2}} + \frac{1}{a_{2m+3+\xi(m)}} \ \le\ \mathcal{G}(\theta).$$
This contradicts \eqref{e7}. Therefore, 
$$\theta\in \left(\frac{1}{a_{2m+3}} + \frac{1}{a_{2m+4}}, \frac{1}{a_{2m+2}} +\frac{1}{a_{2m+3+\xi(m)}}\right].$$

\end{proof}
\section{Application to some sequences}
The goal of this section is to apply (\ref{m1}) to specific sequences. The first one is going to the Fibonacci sequence $(F_n)_{n=1}^\infty$. Let
$a_{n}=F_{n+1}$ then it is easy to see that $(a_n)_{n=0}^\infty$ satisfies all conditions. 
\begin{prop}

The greedy two-term underapproximation $\mathcal{G}(\theta)$ is the best underapproximation out of $\mathcal{A}(\theta, 2)$ if and only if 
$$\theta \ \notin\ \bigcup_{n = 0}^{\infty}\left(\frac{1}{F_{2n+4}}+\frac{1}{F_{2n+5}}, \frac{1}{F_{2n+3}} + \frac{1}{F_{6n+8}}\right],$$
\begin{proof} We want to show that if $\xi(n)$ is the largest nonnegative integer such that
$$F_{2n+3}F_{\xi(n)} +F_{2n+4}F_{\xi(n)+1}\ \le\ {F_{2n+3}F_{2n+4}F_{2n+5}}.$$
then $\xi(n) = 4n+4$
Using the Binet's formula, we have
    \begin{align*}
       &F_{2n+3}F_{2n+4}F_{2n+5}\\
       &\ =\ \frac{(\varphi^{2n+3} - \varphi^{-2n-3})(\varphi^{2n+4} - \varphi^{-2n-4})(\varphi^{2n+5} - \varphi^{-2n-5})}{5\sqrt{5}}\\
       &\ =\ \frac{\varphi^{6n+12} - \varphi^{-6n-12} - ( \varphi^{2n+6} - \varphi^{-2n-6} + \varphi^{2n+4} - \varphi^{-2n-4} + \varphi^{2n+2} - \varphi^{-2n-2})}{5\sqrt{5}} \\
       &\ =\ \frac{F_{6n+12} - F_{2n+6} - F_{2n+4} - F_{2n+2}}{5} \\
       &\ =\ \frac{F_{6n+12} - F_{2n+6} - F_{2n+5} + F_{2n+1}}{5} \\
       &\ =\ \frac{3  F_{6n+9}  + 2 F_{6n+8} - F_{2n+7} + F_{2n+1}}{5}
    \end{align*}
    By (\ref{l1}) $$F_{2n+3}F_{\xi(n)} +F_{2n+4}F_{\xi(n)+1} = F_{2n+4+\xi(n)}.$$ So now it's enough to prove that
    $$F_{6n + 8} \leq \frac{3  F_{6n+9}  + 2 F_{6n+8} - F_{2n+7} + F_{2n+1}}{5}  < F_{6n + 9}.$$ Which is true because
\begin{align*}
F_{6n+8} < \frac{2F_{6n+9}  + 3F_{6n+8}}{5}&\ <\ \frac{3  F_{6n+9}  + 2 F_{6n+8} - F_{2n+7} + F_{2n+1}}{5}\\
&\ <\ \frac{4  F_{6n+9}  + F_{6n+8}}{5}\\
&\ <\ F_{6n+9}
\end{align*}    
    
\end{proof}
\end{prop}

Next, consider the Lucas sequence $(L_n)_{n=0}^\infty$ is defined as $L_0 = 2, L_1 = 1, L_n = L_{n-1} + L_{n-2}$ where $n \ge 2$. Set $a_n = L_{n+2}$. Then it is easy to see that $(a_n)_{n=0}^\infty$ satisfies all conditions. So we will be working with $(L_n)_{n=2}^\infty$.
\begin{prop}

    The greedy two-term underapproximation $\mathcal{G}(\theta)$ is the best underapproximation out of $\mathcal{A}(\theta, 2)$ if and only if 
$$\theta \ \notin\ \bigcup_{n = 0}^{\infty}\left(\frac{1}{L_{2n+5}}+\frac{1}{L_{2n+6}}, \frac{1}{L_{2n+4}} + \frac{1}{L_{6n+11}}\right].$$
\begin{proof}
    We want to show that if $\xi(n)$ is the largest nonnegative integer such that
$$L_{2n+4}F_{\xi(n)} +L_{2n+5}F_{\xi(n)+1}\ \le\ \frac{L_{2n+4}L_{2n+5}L_{2n+6}}{5},$$ then
    $$\xi(n) = 4n+6.$$

The proof is similar to the proof of Proposition 4.1:
\begin{align*}
    \frac{L_{2n+4}L_{2n+5}L_{2n+6}}{5} &\ =\ \frac{(\varphi^{2n+4} + \varphi^{-2n-4})(\varphi^{2n+5} + \varphi^{-2n-5})(\varphi^{2n+6} + \varphi^{-2n-6})}{5}\\
       &\ =\ \frac{L_{6n+15} + L_{2n+7} + L_{2n+5} + L_{2n+3}}{5}\\
       &\ =\ \frac{L_{6n+15} + L_{2n+8} - L_{2n+2}}{5}\\
       &\ =\ \frac{3L_{6n+12} + 2L_{6n+11} + L_{2n+8} - L_{2n+2}}{5}.
\end{align*}
By (\ref{l1}),$$L_{2n+4}F_{\xi(n)} +L_{2n+5}F_{\xi(n)+1} \ = \ L_{2n+4+\xi(n)} .$$
So we need to check that:
$$L_{6n+11} \leq  \frac{3L_{6n+12} + 2L_{6n+11} + L_{2n+8} - L_{2n+2}}{5} < L_{6n+12} ,$$
which is true because 
\begin{align*}
    L_{6n+11} < \frac{2L_{6n+12} + 3L_{6n+11}}{5} &\ <\ \frac{3L_{6n+12} + 2L_{6n+11} + L_{2n+8} - L_{2n+2}}{5}\\
    &\ <\ \frac{4L_{6n+12} + L_{6n+11}}{5}\\
    &\ <\ L_{6n+12} .
\end{align*}
\end{proof}
\end{prop}    

\section{Further investigation}
Inspired by \cite{2}, one possible direction would be to find the Lebesgue measure of all real numbers whose finite greedy underapproximations $\sum_{i=1}^{m}1/F_{n_i}, m \in \mathbb{N},$ are eventually the best among all using reciprocals of Fibonacci numbers.

We would like to investigate the greedy two-term underapproximations using terms from a sequence $(c_n)_{n=0}^{\infty}$, where $(c_n)_{n=0}^{\infty}$ satisfies $c_n = ac_{n-1} + bc_{n-2}, \mbox{ for }a,b \in \mathbb{N} $ or a recurrence of higher order. One of the difficulties in solving this problem is how involved the algebra becomes without nice equations as in Lemma 2.2.

\section{Appendix}
\begin{proof}[Proof of \eqref{e1}]
Fix $m\in \mathbb{Z}$. We induct on $n$. Base cases: for $n = 0$, 
$$F_{-1}F_m +F_0 F_{m+1} \ =\ F_m\ =\ F_{0 +m};$$
for $n = 1$,
$$F_{0}F_m + F_1 F_{m+1}\ =\ F_{1+m}.$$
Inductive hypothesis: suppose that \eqref{e1} holds for $n = k\ge 1$. We have
\begin{align*}
F_{k+1+m}\ =\ F_{k+m} +F_{(k-1)+m}&\ =\ (F_{k-1}F_m + F_{k}F_{m+1}) + (F_{k-2}F_m +F_{k-1}F_{m+1})\\
&\ =\ (F_{k-1}+F_{k-2})F_m + (F_k +F_{k-1})F_{m+1}\\
&\ =\ F_k F_m +F_{k+1}F_{m+1}.
\end{align*}
Therefore, \eqref{e1} holds for all $n\ge 0$. 

We prove \eqref{e1} for $n < 0$. 
Inductive hypothesis: suppose that \eqref{e1} holds for all $n \ge k$ for some $k\le 0$. We have
\begin{align*}
 F_{k-1+m}\ =\ F_{k+1+m}-F_{k+m} &\ =\ (F_{k}F_m+F_{k+1}F_{m+1}) - (F_{k-1}F_m + F_{k}F_{m+1})\\
 &\ =\ (F_k - F_{k-1})F_m + (F_{k+1}-F_k)F_{m+1}\\
 &\ =\ F_{k-2}F_m + F_{k-1}F_{m+1}. 
\end{align*}
This completes our proof. 
\end{proof}


\ \\
\end{document}